\documentclass[11pt]{article}

\usepackage[utf8]{inputenc}
\usepackage[british]{babel}
\usepackage[a4paper, margin=2.5cm]{geometry}
\usepackage{amsmath, amsthm, amssymb, amsfonts}
\usepackage{mathtools}
\usepackage{thmtools}
\usepackage[shortlabels]{enumitem}
\usepackage[font={small, it}]{caption}
\usepackage{microtype}
\usepackage{xcolor}
\usepackage{mathabx}
\usepackage{tikz}
\usepackage{hyperref}
\usepackage{bm}

\newcommand{\cB}{\ensuremath{\mathcal B}}

\newcommand{\cE}{\ensuremath{\mathcal E}}

\newcommand{\cP}{\ensuremath{\mathcal P}}

\newcommand{\eps}{\varepsilon}
\renewcommand{\phi}{\varphi}
\renewcommand{\rho}{\varrho}

\DeclareMathOperator*{\N}{\mathbb{N}}

\DeclarePairedDelimiter\ceil{\lceil}{\rceil}

\let\setminus=\smallsetminus
\let\emptyset=\varnothing

\newcommand{\Gnp}{G_{n, p}}

\declaretheorem[parent=section]{theorem}
\declaretheorem[sibling=theorem]{lemma}

\declaretheorem[sibling=theorem,style=definition]{definition}

\setlist{itemsep=0.1em, topsep=0.1em, parsep=0.1em, partopsep=0.1em}

\hypersetup{
  colorlinks,
  linkcolor={red!60!black},
  citecolor={green!50!black},
  urlcolor={blue!80!black}
}

\colorlet{RoyalRed}{red!70!black}
\definecolor{RoyalBlue}{rgb}{0.25, 0.41, 0.88}
\definecolor{RoyalAzure}{rgb}{0.0, 0.22, 0.66}

\newlength{\bibitemsep}
\newlength{\bibparskip}
\let\oldthebibliography\thebibliography
\renewcommand\thebibliography[1]{%
  \oldthebibliography{#1}%
  \setlength{\parskip}{\bibitemsep}%
  \setlength{\itemsep}{\bibparskip}%
}

\title{The size-Ramsey number of cubic graphs}
\author{
  David Conlon\thanks{Department of Mathematics, California Institute of
  Technology, Pasadena, CA 91125, USA. Email: \texttt{dconlon@caltech.edu}.
  Research supported by NSF Award DMS-2054452.}
  \and
  Rajko Nenadov\thanks{Google Z\"urich. Email: \texttt{rajkon@gmail.com}.}
  \and
  Milo\v{s} Truji\'{c}\thanks{Institute of Theoretical Computer Science, ETH
  Z\"{u}rich, 8092 Z\"{u}rich, Switzerland. Email: \texttt{mtrujic@inf.ethz.ch}.
  Research supported by grant no.\ 200020 197138 of the Swiss National Science
  Foundation.}
}
\date{}

\begin{document}
\maketitle

\begin{abstract}
  We show that the size-Ramsey number of any cubic graph with $n$ vertices is
  $O(n^{8/5})$, improving a bound of $n^{5/3 + o(1)}$ due to Kohayakawa,
  R\"{o}dl, Schacht, and Szemer\'{e}di. The heart of the argument is to show
  that there is a constant $C$ such that a random graph with $C n$ vertices
  where every edge is chosen independently with probability $p \geq C n^{-2/5}$
  is with high probability Ramsey for any cubic graph with $n$ vertices. This
  latter result is best possible up to the constant.
\end{abstract}

\section{Introduction}

We say that a graph $G$ is \emph{Ramsey} for another graph $H$, and write $G
\rightarrow H$, if every colouring of the edges of $G$ with two colours contains
a monochromatic copy of $H$. In this paper, we study the quantity $\hat r(H)$,
the \emph{size-Ramsey number} of $H$, defined as the smallest $m \in \N$ such
that there exists a graph $G$ with $m$ edges which is Ramsey for $H$. As any
sufficiently large complete graph is Ramsey for $H$, this notion is
well-defined.

The study of size-Ramsey numbers was initiated by Erd\H{o}s, Faudree, Rousseau,
and Schelp~\cite{erdHos1978size} in 1978. Their paper already contains the
observation, which they attribute to Chv\'atal, that $\hat r(K_n) =
\binom{r(K_n)}{2}$, where $r(K_n)$ is the usual \emph{Ramsey number}, i.e., the
smallest integer $N$ such that $K_N \rightarrow K_n$. In other words, when $H =
K_n$ one cannot do better than taking $G$ to be the smallest complete graph that
is Ramsey for $H$. A significantly more interesting example is when $H = P_n$,
the path on $n$ vertices, for which Beck~\cite{beck1983size} has shown that
$\hat r(P_n) = O(n)$. Subsequent work has extended this result to many other
familes of graphs, including bounded-degree trees~\cite{friedman1987expanding},
cycles~\cite{haxell1995induced}, and, more recently, powers of paths and
bounded-degree trees~\cite{berger2019size} and long
subdivisions~\cite{draganic_kolutovi}. For some further recent developments, see
\cite{clemens2019size, han2020multicolour, han2019size, kamcev2021size,
letzter2021size}.

Moving away from trees and tree-like graphs, Beck~\cite{beck1990size} asked
whether the size-Ramsey number of every bounded-degree graph is linear in its
order. R\"odl and Szemer\'edi~\cite{rodl2000size} answered this in the negative
by showing that there exists a constant $c > 0$ and, for every $n$, an
$n$-vertex {\em cubic graph} $H$, that is, a graph with maximum degree three,
such that $\hat r(H) \geq n (\log n)^c$. To date, this remains the best known
general lower bound for the size-Ramsey number of bounded-degree graphs. Despite
this, the conjecture of R\"odl and Szemer\'edi~\cite{rodl2000size} that there
exists a constant $\eps > 0$ and, for every $n$, an $n$-vertex cubic graph $H$
such that $\hat r(H) \geq n^{1+\eps}$ is widely believed.

Regarding upper bounds, a result of Kohayakawa, R\"{o}dl, Schacht, and
Szemer\'{e}di~\cite{kohayakawa2011sparse} shows that $\hat r(H) \leq
n^{2-1/\Delta + o(1)}$ for any graph $H$ on $n$ vertices with maximum degree
$\Delta$. Here we improve this bound for cubic graphs.

\begin{theorem}
 \label{thm:main-theorem}
  There exists a constant $K$ such that $\hat r(H) \leq Kn^{8/5}$ for every
  cubic graph $H$ with $n$ vertices.
\end{theorem}

Seeing as there is still a very large gap between the upper and lower bounds for
size-Ramsey numbers of cubic graphs, our result warrants some justification.
The key point is that all previous work on upper bounds for size-Ramsey numbers
has relied upon showing that suitable random graphs (or, in some cases, graphs
derived from random graphs by taking appropriate powers and blow-ups) are, with
high probability,\footnote{We say that a property holds with high probability
(or w.h.p.\ for brevity) if the probability it holds tends to $1$ as $n \to
\infty$.} Ramsey for the required target graph $H$.  Recall that $\Gnp$ stands
for the probability distribution over all graphs with $n \in \N$ vertices in
which each pair of vertices forms an edge independently with probability $p =
p(n) \in (0,1)$. We will use $\Gnp$ interchangeably to describe both this
distribution and an actual graph sampled from it. With this notation, our real
main result is then the following theorem about Ramsey properties of random
graphs.

\begin{theorem} \label{thm:main_gnp}
  There exist $c, K > 0$ such that if $p \ge Kn^{-2/5}$, then, with high
  probability, $\Gnp \rightarrow H$ for every cubic graph $H$ with at most $cn$
  vertices.
\end{theorem}

It is easy to see that Theorem~\ref{thm:main_gnp} implies
Theorem~\ref{thm:main-theorem}, since $\Gnp$ typically has $\Theta(n^2p)$ edges,
by standard concentration inequalities. Moreover, Theorem~\ref{thm:main_gnp} is
optimal: a classic result of R\"{o}dl and Ruci\'{n}ski~\cite{rodl1993lower,
rodl1995threshold} shows that if $p = o(n^{-2/5})$, then, with high probability,
$\Gnp$ is not Ramsey for $K_4$. Hence, the statement of
Theorem~\ref{thm:main-theorem} is the most one can get out of using vanilla
random graphs (that is, without modifying them further). We will come back to
this point, and a further discussion on the limits of our method, in
Section~\ref{sec:concluding}.

The rest of the paper is organised as follows. In the next section, we give a
high-level overview of our argument. In Section~\ref{sec:preliminaries}, we
collect several results about random graphs and sparse regular pairs that will
be needed in the proof. Section~\ref{sec:strategy-trees} contains our two main
building blocks, which will allow us to thread trees and cycles through
prescribed sets of vertices. In Section~\ref{sec:main-proof}, we then combine
these building blocks with a decomposition result for cubic graphs to complete
the proof of Theorem~\ref{thm:main_gnp}. Finally, as mentioned above, we will
discuss the limits of our method and the potential next steps in
Section~\ref{sec:concluding}.

\section{Overview of the proof}
\label{sec:proof_overview}

Assume that the host graph $\Gamma \sim \Gnp$ satisfies a number of the
properties that typically hold in random graphs, such as expansion, no dense
spots, concentration of degrees and number of edges in subsets, etc. In their
proof, Kohayakawa, R\"{o}dl, Schacht, and
Szemer\'{e}di~\cite{kohayakawa2011sparse} start by showing that in any red/blue
colouring of the edges of $\Gamma$, one can find disjoint sets $V_1, \dotsc,
V_{20} \subseteq V(\Gamma)$, each of order $\alpha n$ for some $\alpha > 0$,
such that each pair $(V_i, V_j)$ is \emph{$(\eps, p)$-regular} with at least
$|V_i||V_j|p/3$ edges between $V_i$ and $V_j$ in one of the colours, say red.
That is, the density of edges between any two sufficiently large sets $U_i
\subseteq V_i$ and $U_j \subseteq V_j$ is roughly the same as the density
between $V_i$ and $V_j$, with a discrepancy of at most $\eps p$ (see
Definition~\ref{def:regularity} below). This is a standard step in the
regularity method and we refer the reader to \cite{chvatal1983ramsey,
kohayakawa2011sparse} for more details.

Suppose, therefore, that we have a collection of large sets of vertices such
that the distribution of red edges between any two of them is fairly uniform. We
wish to use this structure to show that $R$, the subgraph of $\Gamma$ consisting
of all the red edges, contains any particular cubic graph $H$ on at most $cn$
vertices. The strategy used in \cite{kohayakawa2011sparse} at this point follows
an idea of Alon and F\"{u}redi~\cite{alon1992spanning}: split the vertex set of
$H$ into, say, 20 independent sets and then embed these sets one at a time using
Hall's matching criteria, with the $i$-th set being embedded into $V_i$. When we
come to embed the last such set, every remaining vertex of degree three has to
be mapped into the common neighbourhood of three previously embedded vertices.
However, if $p = o(n^{-1/3})$, three typical vertices in the random graph will
have an empty common neighbourhood. It is for this reason that the methods of
\cite{kohayakawa2011sparse} break down at this point.

To circumvent this issue, we borrow an idea from the work of the first two
authors together with Ferber and \v{S}kori\'{c}~\cite{conlon2017almost}. Assume
that $H$ is connected and remove from it an induced cycle of length at least
four. This leaves us with a $2$-degenerate graph $H'$. That is, we can order the
vertices of $H'$ in such a way that every vertex has at most two of its
neighbours preceding it. One might then hope that $p \gg n^{-1/2}$ is sufficient
to embed $H'$. Having embedded $H'$, we then need to replace the deleted cycle,
which we suppose for illustration is a $C_4$. As the graph is cubic, each vertex
$v$ of such a $C_4$ already has at most one embedded neighbour, so there will be
a \emph{candidate set} of order roughly $np$ in which we can embed $v$. If all
four candidate sets are disjoint and span four $(\eps, p)$-regular pairs of
density $\Theta(p)$ in the correct cyclic order, then a result of Gerke,
Kohayakawa, R\"{o}dl, and Steger~\cite{gerke2007small} implies that we can find
the desired copy of $C_4$ between these sets, provided $p \gg (np)^{-2/3}$ (see
Lemma~\ref{lem:KLR} below). Rearranging, this gives $p \gg n^{-2/5}$ ---
precisely the bound promised by Theorem~\ref{thm:main_gnp}. Note that,
crucially, the cycle we embed at the end is of length at least four. If it were
a triangle instead, we would need $p \gg (np)^{-1/2}$ in the last step, which
leads back to the bound $p \gg n^{-1/3}$.

In practice, our actual approach takes the idea of partitioning even further.
Instead of taking out one cycle and relying on the fact that the remaining graph
$H'$ is $2$-degenerate, we partition $H$ (which we may assume contains no
$K_4$'s, since they can be set aside and dealt with separately) into
\emph{blocks}: first removing a maximal collection of disjoint induced cycles of
length at least four and then partitioning what remains into induced paths (see
Lemma~\ref{lem:decomposition} below). Moreover, these blocks $B_1, \dotsc, B_t$
can be placed in a `$1$-degenerate ordering', meaning that each vertex in $B_i$
has at most one neighbour in $B_{1} \cup \dotsb \cup B_{i-1}$ for every $i \in
\{2, \dotsc, t\}$. We then find a copy of $H$ by embedding one whole block at a
time. Crucially, whenever we are about to embed a block $B_i$, every $v \in B_i$
has at most one previously embedded neighbour, so the candidate set for $v$ has
order $\Omega(np)$. This is large enough that the regularity property is
inherited by any relevant pair of candidate sets, which then allows us to embed $B_i$, which is either a path
or a cycle, in one sweep (see Lemmas~\ref{lem:tree-embedding-lemma}
and~\ref{lem:cycle-embedding-lemma} below).

\section{Preliminaries}\label{sec:preliminaries}

In this section, we collect several results about random graphs that we will
need, with a particular focus on the properties of regular pairs in random
graphs. For a thorough, although now somewhat outdated, treatment of the latter
topic, we refer the reader to the survey of Gerke and
Steger~\cite{gerke2005sparse}. Though many of the results hold in greater
generality, we have tailored the statements towards their later use in the proof
of Theorem~\ref{thm:main_gnp}. Where statements depart from their usual form in
the literature, we also give, at the very least, a sketch of the proof.

We begin with a standard concentration result, which easily follows from
combining Chernoff's inequality (see, e.g.,
\cite[Corollary~2.3]{janson2011random}) with a union bound. Here and throughout
we write $\hat N_G(X, Y)$ for the set of \emph{common neighbours} of the
vertices from $X$ in $Y$, that is, $\hat N_G(X, Y) := Y \cap \bigcap_{x \in X}
N_G(x)$ and $\hat N_G(X) = \hat N_G(X, V(G) \setminus X)$.

\begin{lemma}\label{lem:common-neighbours-expansion}
  For every $d \in \N$ and $\delta \in (0, 1)$, there exists a positive constant
  $K$ such that, for $p \geq (K\log n/n)^{1/d}$, the random graph $\Gamma \sim
  \Gnp$ w.h.p.\ has the following property. For every family of disjoint
  $d$-sets $\cP \subseteq \binom{V(\Gamma)}{d}$ of size $|\cP| \leq \delta/p^d$,
  \[
    \Big| \bigcup_{S \in \cP} \hat N_\Gamma(S) \Big| = (1 \pm \delta)|\cP|np^d.
  \]
\end{lemma}

\subsection{Properties of $(\eps, p)$-regular pairs}

We have already referred to $(\eps, p)$-regular pairs several times. The formal
definition is as follows.

\begin{definition}\label{def:regularity}
  Let $G$ be a graph and let $V_1, V_2 \subseteq V(G)$ be disjoint subsets. We
  say that the pair $(V_1, V_2)$ is \emph{$(\eps, p)$-regular} for some $0 <
  \eps, p \le 1$ if, for every $U_1 \subseteq V_1$, $U_2 \subseteq V_2$ with
  $|U_1| \ge \eps|V_1|$, $|U_2| \geq \eps|V_2|$, we have
  \[
    \big|d_G(U_1, U_2) - d_G(V_1,V_2)\big| \leq \eps p,
  \]
  where $d_G(A, B) = e_G(A, B)/(|A||B|)$ denotes the edge density of a given
  pair.
\end{definition}

It follows from the definition that if $d = d_G(V_1, V_2) = \Theta(p)$ and
$\eps$ is sufficiently small, then there cannot be more than $\eps |V_1|$
vertices in $V_1$ which have fewer than, say, $d|V_2|/2$ neighbours in $V_2$.
The next result shows that we can even find large subsets $V_1' \subseteq V_1$
and $V_2' \subseteq V_2$ such that each vertex in $V_1'$ has at least $d|V_2|/2$
neighbours in $V_2'$ and vice versa and, moreover, that this property can be
achieved for many pairs simultaneously.

\begin{lemma}\label{lem:cleanup}
  For every $\Delta \in \N$ and $\gamma > 0$, there exists $\eps_0 > 0$ such
  that the following holds for any $0 < \eps \le \eps_0$ and $p \in (0,1)$. Let
  $H$ be a graph with maximum degree $\Delta$ and let $\{V_i\}_{i \in V(H)}$ be
  a family of subsets of some graph $G$ such that $(V_i,V_j)$ is
  $(\eps,p)$-regular of density $d \geq \gamma p$ (with respect to $G$) for
  every $ij \in H$. Then, for every $i \in V(H)$, there exists $V_i' \subseteq
  V_i$ of order $|V_i'| \ge (1-\Delta\eps)|V_i|$ such that $\deg_G(v, V_j') \geq
  d|V_j|/2$ for every $v \in V_i'$ and all $ij \in H$.
\end{lemma}

\begin{proof}[Sketch of the proof.]
  Set $B_i = \emptyset$ for every $i \in V(H)$ and repeat the following process:
  as long as there is an edge $ij \in H$ and a vertex $v \in V_i \setminus B_i$
  which has fewer than $d|V_j|/2$ neighbours in $V_j \setminus B_j$, add $v$ to
  $B_i$. Suppose that at some point one of the sets $B_i$ becomes larger than
  $\Delta\eps|V_i|$ and, once this happens, we terminate the process. By the
  pigeonhole principle, there must be a subset $B_i' \subseteq B_i$ of order
  $|B_i'| \ge \eps|V_i|$ and an edge $ij \in H$ such that every $v \in B_i'$
  satisfies $\deg_G(v, V_j \setminus B_j) < d|V_j|/2$. This implies that the
  density of the pair $(B_i', V_j \setminus B_j)$ is less than $d/2$. On the
  other hand, as $|V_j \setminus B_j| \ge (1-\Delta\eps)|V_j|$, the $(\eps,
  p)$-regularity property tells us that this density is close to $d$, a
  contradiction. In particular, the procedure terminates with each $B_i$ being
  of order at most $\Delta\eps|V_i|$ and the statement of the lemma follows.
\end{proof}

\subsubsection{Regularity inheritance}
\label{sec:sparse_reg}

The following lemma is usually referred to as the \emph{slicing lemma} and follows
directly from the definition of $(\eps, p)$-regularity.

\begin{lemma}\label{lem:large-subset-reg-inheritance}
  Let $0 < \eps_1 < \eps_2 \leq 1/2$, $p \in (0, 1)$, and let $(X, Y)$ be an
  $(\eps_1, p)$-regular pair. Then any two subsets $X' \subseteq X$ and $Y'
  \subseteq Y$ of order $|X'| \geq \eps_2|X|$ and $|Y'| \geq \eps_2|Y|$ form an
  $(\eps_1/\eps_2, p)$-regular pair of density $d(X, Y) \pm \eps_1 p$.
\end{lemma}

In other words, sufficiently large subsets of a regular pair again induce a
regular pair. This then allows us to take subsets of our regular pairs and yet
still assume that they are $(\eps, p)$-regular with $\eps$ sufficiently small, a
fact that we will use implicitly in our main proof.

The next lemma captures a key feature of sparse regularity, that, for subgraphs
of random graphs, the regularity property is typically inherited between
neighbourhoods of vertices.

\begin{lemma}
  \label{lem:typical_vertices}
  For all $\eps', \alpha, \gamma, \delta > 0$, there exist $\eps_0 =
  \eps_0(\eps',\gamma,\delta)$ and $K = K(\eps',\alpha,\gamma)$ such that, for
  every $0 < \eps \le \eps_0$ and $p \ge K(\log n/n)^{1/2}$, the random graph
  $\Gamma \sim \Gnp$ w.h.p.\ has the following property.

  Suppose $G \subseteq \Gamma$ and $V_1, V_2 \subseteq V(\Gamma)$ are disjoint
  subsets of order $\tilde n = \alpha n$ such that $(V_1, V_2)$ is
  $(\eps,p)$-regular of density $d \ge \gamma p$ with respect to $G$. Then there
  exists $B \subseteq V(\Gamma)$ of order $|B| \le \delta\tilde n$ such that for
  each $v, w \in V(\Gamma) \setminus (V_1 \cup V_2 \cup B)$ (not necessarily
  distinct) the following holds: for any two subsets $N_v \subseteq N_\Gamma(v,
  V_1)$ and $N_w \subseteq N_\Gamma(w, V_2)$ of order $\tilde n d/4$, both
  $(N_v, V_2)$ and $(N_v, N_w)$ are $(\eps', p)$-regular of density $(1 \pm
  \eps')d$ with respect to $G$.
\end{lemma}

To prove Lemma~\ref{lem:typical_vertices}, we need the following lemma of \v
Skori\'c, Steger, and Truji\'c~\cite{vskoric2018local}, itself based on an
earlier result of Gerke, Kohayakawa, R\"{o}dl, and Steger~\cite{gerke2007small}.
It enables us to say something about regularity inheritance for subsets of order
$o(n)$, a regime in which Lemma~\ref{lem:large-subset-reg-inheritance} tells us
nothing.

\begin{lemma}[Corollary~3.5 in \cite{vskoric2018local}]
  \label{lem:skoric_trujic}
  For all $0 < \beta, \eps', \gamma < 1$, there exist positive constants $\eps_0
  = \eps_0(\beta,\eps',\gamma)$ and $D = D(\eps')$ such that, for every $0 <
  \eps \leq \eps_0$ and $p = p(n) \in (0,1)$, the random graph $\Gamma \sim
  \Gnp$ w.h.p.\ has the following property. Suppose $G \subseteq \Gamma$ and
  $(V_1, V_2)$ is an $(\eps,p)$-regular pair of density $d \ge \gamma p$ with
  respect to $G$. Then, for all $q_1, q_2 \geq Dp^{-1}\log n$, there are at
  least
  \[
    \big(1-\beta^{\min\{q_1,q_2\}}\big) \binom{|V_1|}{q_1} \binom{|V_2|}{q_2}
  \]
  sets $Q_i \subseteq V_i$ of order $|Q_i| = q_i$, $i \in \{1, 2\}$, which induce
  an $(\eps',p)$-regular pair of density $(1 \pm \eps')d$ with respect to $G$.
\end{lemma}

Intuitively, Lemma~\ref{lem:skoric_trujic} tells us that if we have an $(\eps,
p)$-regular pair $(V_1, V_2)$ with respect to a subgraph of $\Gnp$ and two
vertices $v$ and $w$ which have neighbourhoods of order $\Theta(np)$ in $V_1$
and $V_2$, respectively, then, for $p = \Omega((\log n/n)^{1/2})$, it is
extremely unlikely that these two neighbourhoods do not inherit regularity. In
fact, it is so unlikely that a union bound over all possible situations suffices
to prove Lemma~\ref{lem:typical_vertices}.

\begin{proof}[Proof of Lemma~\ref{lem:typical_vertices}]
  We only prove that there exists a $B$ which ensures that $(N_v, N_w)$ is
  always $(\eps', p)$-regular with the required density. The other case follows
  from similar calculations, so we omit it.

  For two disjoint subsets $V_1, V_2 \subseteq V(\Gamma)$ of order $\tilde n$,
  let $\cE(V_1, V_2)$ be the event that there exists $G \subseteq \Gamma$ for
  which the conclusion of the lemma fails for this particular choice of $V_1$
  and $V_2$. We will show that $\Pr[\cE(V_1, V_2)] \le \exp(-c\tilde n^2 p)$
  for some $c > 0$, which is more than enough to beat the union bound over all
  choices of $V_1$ and $V_2$.

  Suppose $G$ is a witness for $\cE(V_1, V_2)$. Consider the following process:
  start with $\cB = \emptyset$ and $B = \emptyset$ and, as long as there exist
  (not necessarily distinct) vertices $v, w \in V(G) \setminus (B \cup V_1 \cup
  V_2)$ which violate the desired property, add the pair $(v, w)$ to $\cB$ and
  add $v$ and $w$ to $B$. As soon as $|B|$ becomes at least $\delta \tilde n$, we stop the procedure. Note that the procedure must go on at least this long, by the
  choice of $G$. Then $\cB$ contains at least $\delta\tilde n/2$ and at most
  $\delta\tilde n$ pairs.

  Therefore, if $G$ is a witness for $\cE(V_1, V_2)$, then there exists a set of
  pairs of vertices $\cB$ of size $\delta\tilde n/2 \le |\cB| \le \delta\tilde
  n$ such that all elements of $\cB$ are pairwise disjoint and, for every $(v,
  w) \in \cB$, there exist $N_v \subseteq N_\Gamma(v, V_1)$ and $N_w \subseteq
  N_\Gamma(w, V_2)$, both of order $\tilde n d/4$, such that $(N_v, N_w)$ does
  not span an $(\eps', p)$-regular pair of density $(1 \pm \eps')d$ with respect
  to $G$. Let us bound the probability that such a configuration exists in
  $\Gnp$.

  We first expose the edges between $V_1$ and $V_2$ in $\Gamma$. Choose a subset
  $E' \subseteq E_\Gamma(V_1, V_2)$ of these edges such that $(V_1, V_2)$ is
  $(\eps, p)$-regular of density $d \geq \gamma p$ with respect to $E'$.
  Conditioning on $|E_\Gamma(V_1, V_2)| \le 2 \tilde n^2 p$ (which we may assume
  holds w.h.p.~by a standard application of Chernoff's inequality and the union
  bound), there are at most $2^{2 \tilde n^2 p}$ choices for $E'$. Next, there
  are at most $\delta\tilde n \cdot (n!)^2 < 2^{3 n \log n}$ choices for $\cB$.
  Finally, for each $(v, w) \in \cB$, choose the `bad' subsets $N_v \subseteq
  V_1$ and $N_w \subseteq V_2$ of order $x = \tilde nd/4$. As $(V_1,V_2)$
  satisfies the conclusion of Lemma~\ref{lem:skoric_trujic} and $\tilde n d/4
  \geq D_{\ref{lem:skoric_trujic}}(\eps') p^{-1}\log n$ by the assumption on $p$
  from the statement of the lemma and the fact that $K$ is sufficiently large,
  there are at most
  \[
    \bigg( \beta^x \binom{\tilde n}{x}^2 \bigg)^{|\cB|}
  \]
  such choices. But we also need that each such set $N_v$ lies in the
  neighbourhood of $v$ in $\Gamma$ and similarly for $N_w$ and $w$. Using the
  fact that any two pairs in $\cB$ are disjoint and, within each pair, we
  consider neighbours into disjoint sets $V_1$ and $V_2$, the probability of
  this happening is exactly $p^{2|\cB|x}$. Putting all this together, we get
  that
  \[
    \Pr[\cE(V_1, V_2)] \le 2^{2 \tilde n^2 p} \cdot 2^{3n \log n} \cdot
    \beta^{|\cB|\tilde n d / 4} \binom{\tilde n}{\tilde n d / 4}^{2|\cB|} \cdot
    p^{2|\cB| \tilde n d / 4}.
  \]
  Using the standard bound $\binom{a}{b} \le (\frac{ea}{b})^b$, we conclude that
  by taking $\beta$ sufficiently small we can make this be at most $\exp(-c
  \tilde n^2 p)$, as desired.
\end{proof}

\subsubsection{The K\L R conjecture in random graphs}

The following result from \cite{conlon2014klr} gives sufficient conditions for
the existence of a small, fixed graph between an appropriate collection of
$(\eps, p)$-regular pairs in a random graph. This result also follows from a
celebrated conjecture of Kohayakawa, \L uczak, and
R\"odl~\cite{kohayakawa1997klr} (the so-called \emph{K\L R conjecture}), which
was fully resolved by Balogh, Morris, and Samotij~\cite{balogh2015independent}
and, independently, Saxton and Thomason~\cite{saxton2015hypergraph} (though see
the recent paper~\cite{nenadov2021klr} for another proof). However, we will only
invoke the lemma when $H$ is either a cycle or $K_4$, both of which were known
before the full conjecture was proved (see~\cite{gerke2007small}
and~\cite{gerke2007k}, respectively).

\begin{lemma}[The K\L R conjecture]\label{lem:KLR}
  For every graph $H$ and every $\gamma > 0$, there exist $\eps_0, K > 0$ such
  that, for every $0 < \eps \leq \eps_0$ and $p \ge Kn^{-1/m_2(H)}$, where
  \[
    m_2(H) = \max\Big\{ \frac{e(F)-1}{v(F)-2} : F \subseteq H, v(F) \ge 3
    \Big\},
  \]
  the random graph $\Gamma \sim \Gnp$ w.h.p.\ has the following property.

  Suppose $G \subseteq \Gamma$ and $\{V_i\}_{i \in V(H)}$ is a family of
  disjoint subsets of $V(G)$, each of order $\tilde n \ge \max\{ (K/p)^{m_2(H)},
  K\log n/p \}$. Suppose also that, for each $ij \in H$, the pair $(V_i, V_j)$
  is $(\eps, p)$-regular of density at least $\gamma p$ with respect to $G$.
  Then there exists a copy of $H$ in $G$ which maps each vertex $i$ to $V_i$.
\end{lemma}

\begin{proof}[Sketch of the proof]
  Theorem 1.6 in \cite{conlon2014klr} gives $K > 1$ such that the conclusion of
  the lemma holds with probability at least $1 - \exp(- b n^2 p)$ in the special
  case where $\tilde n = n/v(H)$, for some constant $b > 0$ depending on $H$ and
  $\gamma$. As every subgraph of $\Gnp$ with $s$ vertices is distributed as
  $G_{s, p}$, we can apply the above to such a subset as long as $p \ge K
  s^{-1/m_2(H)}$ or, equivalently, $s \geq (K/p)^{m_2(H)}$. In particular, a
  subgraph of $\Gnp$ induced by a vertex subset $S$ of order $s \ge
  (K/p)^{m_2(H)}$ fails to have the required property in the case $\tilde n = s
  / v(H)$ with probability at most $\exp(-b s^2p)$. By the assumption $\tilde n
  \ge K\log n/p$, and as we may assume $K$ is sufficiently large in terms of
  $b$, this probability is at most $n^{-2s}$, which is sufficient to take a
  union bound over all possible choices of $S$. As every $(\eps, p)$-regular
  constellation with each set of order $\tilde n$ naturally lies in a subset $S$
  of order $v(H) \tilde n$, the claimed statement follows.
\end{proof}

\section{Trees and cycles through prescribed sets}
\label{sec:strategy-trees}

In this section, we provide the two main building blocks used in the proof of
Theorem~\ref{thm:main_gnp}. Both results are stated in greater generality that
what is needed for Theorem~\ref{thm:main_gnp}, with the goal of making further
claims in Section~\ref{sec:concluding} more transparent.

Our first building block says that, under appropriate conditions, one can embed
bounded-degree trees so that each vertex is mapped into a prescribed set. The
embedding strategy used both here and in the main theorem, which we dub
\emph{first-free-bucket embedding}, originates from the second author's PhD
thesis~\cite{rajko_thesis}.  While the bound on $p$ here could be lowered to,
say, $n^{-1/2} \log^3 n$, we have decided not to complicate the proof any
further.

\begin{lemma}\label{lem:tree-embedding-lemma}
  For every $\alpha, \gamma, \xi > 0$ and $\Delta \in \N$, there exist $c =
  c(\alpha,\gamma,\xi,\Delta), \eps_0 = \eps_0(\gamma,\Delta) > 0$ such that,
  for every $0 < \eps \leq \eps_0$ and $p \geq n^{-1/2+\xi}$, the random graph
  $\Gamma \sim \Gnp$ w.h.p.\ has the following property.

  Let $T$ be a tree with $t \leq cn$ vertices and maximum degree $\Delta$. Let
  $G \subseteq \Gamma$ and, for each $v \in V(T)$, let $s_v \in V(G)$ be a
  specific vertex such that every $s \in V(G)$ is chosen at most $\Delta$ times.
  Then, for any collection of subsets $N_v \subseteq N_G(s_v)$ of order $\alpha
  n p$ such that $(N_v, N_w)$ is $(\eps, p)$-regular of density $d \ge \gamma p$
  (with respect to $G$) for each $vw \in T$, there exists a copy of $T$ in $G$
  which maps each $v \in V(T)$ to $N_v$.
\end{lemma}

\begin{proof}
  By applying Lemma~\ref{lem:cleanup} and renaming the sets if necessary, we may
  assume that each vertex in $N_v$ has at least $d|N_w|/2$ neighbours in $N_w$
  for each $vw \in T$ and $|N_v| \ge \tilde n := 0.9 \alpha np$.

  Let $U = \bigcup_{v \in V(T)} N_v$ and take an equipartition $U = U_0 \cup U_1
  \cup \dotsb \cup U_z$ uniformly at random, where $z = \ceil{1/\xi}$. Let
  $N_v^j := U_j \cap N_v$ for all $v \in V(T)$ and $j \in \{0, \dotsc, z\}$. By
  Chernoff's inequality and the union bound, there exists a choice for the $U_j$
  such that every vertex in $N_v$ has at least $d|N_w|/(4z)$ neighbours in each
  $N_w^j$ for all $vw \in T$. In other words, for every $u \in N_v$, we have
  \begin{equation}\label{eq:min_degree}
    |\hat N_G(\{u, s_w\}, N_w^j)| = |N_G(u, N_w^j)| \ge d \tilde n/(4z).
  \end{equation}

  Let $\{v_1, \dotsc, v_t\}$ be an ordering of the vertices of $T$ such that
  each $v_i$, for $i \ge 2$, has exactly one preceding neighbour and denote this
  neighbour by $a_i$. For brevity, rename $N_{v_i}^j$ as $V_i^j$ and $s_{v_i}$
  as $s_i$ and, for every $i \in [t]$, let $T^{< i} := \{v_1, \dotsc,
  v_{i-1}\}$. We construct the desired copy of $T$ by defining an embedding
  $\phi$ as follows:
  \begin{enumerate}[label=(E\arabic*), leftmargin=3em]
    \item assign $\phi(v_1) := x$ for an arbitrary vertex $x \in V_1^0$,
    \item\label{path-embedding-step} for every $i \geq 2$, sequentially, take an
      arbitrary vertex $x$ from the first (smallest $j \in \{0, \dotsc, z\}$)
      non-empty set $N_G(\phi(a_i), V_i^j) \setminus \phi(T^{< i})$ and assign
      $\phi(v_i) := x$, or
    \item\label{terminate} if all such sets are empty, terminate.
  \end{enumerate}
  If the process never reaches~\ref{terminate}, the embedding $\phi$ gives the
  desired copy of $T$. Therefore, suppose towards a contradiction that there is
  some $i \in [t]$ for which the process enters~\ref{terminate}, that is, such
  that $N_G(\phi(a_i), V_i^j) \setminus \phi(T^{< i}) = \varnothing$ for all $j
  \in \{0, \dotsc, z\}$.

  Let $X_j := U_j \cap \phi(T^{< i})$, for all $j \in \{0, \dotsc, z\}$, be
  the set of vertices in $U_j$ `taken' by the images of $\{v_1, \dotsc, v_{i -
  1}\}$ under $\phi$. We claim that
  \begin{equation}\label{eq:star}
    |X_j| \leq \frac{c_j n}{(np^2)^j}
    \tag{$\star$}
  \end{equation}
  for constants $c_j$ with $c_0 = c = \alpha^2\gamma^2/(10^6\Delta z^2)$ and
  $c_j = 400 \Delta z c_{j-1}/(\alpha\gamma)$ for all $j \geq 1$.

  Under this assumption, since $(np^2)^z \geq n^2$, we get that $X_z =
  \varnothing$ and thus $N_G(\phi(a_i), V_i^z) \setminus \phi(T^{< i}) \neq
  \varnothing$, contradicting our assumption that~\ref{terminate} happens at
  step $i$. Therefore, in order to complete the proof, it only remains to show
  \eqref{eq:star}.

  Clearly, $|X_0| \leq c_0n$ as $t \leq c n$. Consider now the smallest $j \in
  [z]$ for which \eqref{eq:star} is violated. Without loss of generality, we may
  assume that $|X_j| = \ceil{c_j n/(np^2)^j}$ (if not, we can take a subset of $X_j$
  of precisely that order). Let $I \subseteq [i-1]$ be the set of indices of
  vertices $v_k \in T^{< i}$ that are embedded into $X_j$, i.e.,
  \[
    I := \{ k \in [i-1] : \phi(v_k) \in X_j \}.
  \]
  Since $\Delta(T) \leq \Delta$ and each $s_k \in V(G)$ is chosen at most
  $\Delta$ times, there exists (by greedily taking indices) $I' \subseteq I$ of
  size at least $|I|/(4\Delta)$ such that the sets $\{\phi(a_k), s_k\}$ are
  pairwise disjoint for all $k \in I'$. For simplicity in our notation, we will
  assume that $I$ already has this property, noting that $|I| \geq
  |X_j|/(4\Delta)$. For every $k \in I$, we set $\cP_k = \{\phi(a_k), s_k\}$.

  Since, for every $k \in I$, the vertex $v_k$ is embedded into $X_j$ and not
  $X_{j - 1}$, we have
  \[
    |X_{j - 1}| \geq \Big| \bigcup_{k \in I} \hat N_G(\cP_k, V_k^{j - 1}) \Big|.
  \]
  Our goal is to show that this implies $X_{j-1}$ is larger than what is claimed
  in \eqref{eq:star}, contradicting our assumption that $j$ is the smallest
  index violating \eqref{eq:star}.

  As $\hat N_G(\cP_k, V_k^{j - 1}) \subseteq \hat N_G(\cP_k) \subseteq \hat
  N_\Gamma(\cP_k)$, we have
  \begin{equation}\label{eq:j-large-mess}
    \Big| \bigcup_{k \in I} \hat N_G(\cP_k, V_k^{j - 1}) \Big| \geq \Big|
    \bigcup_{k \in I} \hat N_\Gamma(\cP_k) \Big| - \Big| \bigcup_{k \in I}
    \big(\hat N_\Gamma(\cP_k) \setminus \hat N_G(\cP_k, V_k^{j - 1})\big) \Big|.
  \end{equation}
  Note that $|I| \leq |X_j|$ and that $|X_j| \leq \lceil c_1/p^2 \rceil$ for $j
  = 1$ and $|X_j| = o(1/p^2)$ for $j \geq 2$. Hence, by
  Lemma~\ref{lem:common-neighbours-expansion} with $d = 2$, $\delta = 2 c_1$,
  and $\cP = \{\cP_k\}_{k \in I}$, the first term on the right-hand side of
  \eqref{eq:j-large-mess} can be bounded by
  \[
    \Big| \bigcup_{k \in I} \hat N_\Gamma(\cP_k) \Big| \geq (1 - \delta)|I|np^2.
  \]
  Lemma~\ref{lem:common-neighbours-expansion} also gives $|\hat N_\Gamma(\cP_k)|
  \leq (1 + \delta)np^2$. Combining the previous two bounds with
  \eqref{eq:min_degree}, we get that
  \[
    |X_{j - 1}| \geq (1 - \delta) |I| np^2
    - |I| \big( (1 + \delta)np^2 - d \tilde n / (4z) \big) \geq
    |I| d \tilde n / (8z) > |I| \alpha\gamma np^2/(100z),
  \]
  where we used that $c_1 \leq \alpha\gamma/(300z)$. Finally, from $|I| \geq
  |X_j|/(4\Delta)$ and $|X_j| = \ceil{c_j n/(np^2)^j}$, we arrive at
  \begin{equation*} \label{eq:contradiction}
    |X_{j-1}| > \frac{c_j n}{4\Delta (np^2)^j} \cdot \frac{\alpha\gamma
    np^2}{100z} \geq \frac{c_{j-1} n}{(np^2)^{j-1}},
  \end{equation*}
  which contradicts our assumption that $j$ was the smallest index violating
  \eqref{eq:star}.
\end{proof}

The following lemma gives the same statement when $T$ is a cycle $C_t$ instead
of a tree. Here, the bound on $p$ comes from the fact that we apply
Lemma~\ref{lem:KLR} with $\tilde n = \Theta(np)$, for which we need $p =
\Omega\big((np)^{-1/m_2(C_t)}\big)$.

\begin{lemma}\label{lem:cycle-embedding-lemma}
  For every $\alpha, \gamma > 0$ and $\ell \ge 3$, there exist $\eps_0 = \eps_0
  (\gamma, \ell), c = c(\alpha, \gamma, \ell), K = K(\alpha,\gamma,\ell) > 0$
  such that the statement of Lemma~\ref{lem:tree-embedding-lemma} holds if $T$
  is a cycle of length $t \in [\ell, cn]$ and $p \ge Kn^{-(\ell-2)/(2\ell-3)}$.
\end{lemma}

\begin{proof}[Sketch of the proof.]
  We first deal with the case where $T$ is a cycle of length $t \in
  [\ell,4\ell]$. Greedily take $N_v' \subseteq N_v$, each of order $|N_v'| =
  |N_v| / (2t)$, such that they are pairwise disjoint. By
  Lemma~\ref{lem:large-subset-reg-inheritance}, these sets are $(\eps',
  p)$-regular with density close to $d$ and so Lemma~\ref{lem:KLR} implies the
  existence of the desired copy of $T$.

  The case $t > 4\ell$ requires a bit more work. Let us denote the vertices of
  $T$ by $v_1, \dotsc, v_t$, in the natural order. Again, for each $i \in \{1,
  \dotsc, \ell+1\} \cup \{t, \dotsc, t-\ell+1\}$, choose a subset $N_{v_i}'
  \subseteq N_{v_i}$ of order $|N_{v_i}'| = |N_{v_i}|/(4\ell)$ such that they are
  pairwise disjoint. Let $V'$ be the union of all these sets and, for every
  other $v$, set $N_v' = N_v \setminus V'$. As
  $(np^2)^\ell \gg np$ by the assumption on $p$, one can expect that the
  $\ell$-th neighbourhood of each vertex $v \in N_{v_1}'$ contains almost all
  vertices in both $N_{v_{\ell+1}}'$ and $N_{v_{t-\ell+1}}'$. A minor
  modification of \cite[Corollary~2.5]{draganic2020size} gives precisely this.
  Namely, it shows that there exists a vertex $v \in N_{v_1}'$ such that, for
  all but $\eps |N_{v_{\ell+1}}'|$ vertices $w \in N_{v_{\ell+1}}'$, there
  exists a path from $v$ to $w$ (in $G$) with one vertex in each of $N_{v_2}',
  \dotsc, N_{v_\ell}'$ and similarly for $N_{v_{t-\ell+1}}'$ with paths through
  $N_{v_t}', \dotsc, N_{v_{t-\ell+2}}'$. Let us denote the set of such vertices
  reachable from $v$ by $N_{v_{\ell+1}}''$ and $N_{v_{t-\ell+1}}''$. Now we can
  apply Lemma~\ref{lem:tree-embedding-lemma} to find a path with one vertex in
  each of $N_{v_{\ell+1}}'', N_{v_{\ell+2}}', \dotsc, N_{v_{t-\ell}}',
  N_{v_{t-\ell+1}}''$. Together with the paths to $v$ from $N_{v_{\ell+1}}''$
  and $N_{v_{t-\ell+1}}''$, this forms the desired cycle.
\end{proof}

\section{Proof of Theorem~\ref{thm:main_gnp}}
\label{sec:main-proof}

Following the overview in Section~\ref{sec:proof_overview}, we first prove
a decomposition result for cubic graphs.

\begin{lemma} \label{lem:decomposition}
  Let $H$ be a connected cubic graph which is not isomorphic to $K_4$. Then
  there exists a partition $V(H) = B_1 \cup \dotsb \cup B_t$, for some $t$ which
  depends on $H$, such that the following hold:
  \begin{itemize}
    \item the subgraph of $H$ induced by each $B_i$ is either a path (of any
      length, even $0$) or a cycle of length at least four;
    \item for every $i \in \{2, \dotsc, t\}$, each vertex in $B_i$ has at most
      one neighbour in $B_{1} \cup \dotsb \cup B_{i-1}$.
    \end{itemize}
\end{lemma}

\begin{proof}
  For simplicity in the proof, we specify the $B_i$'s in reverse order: for
  every $i \in \{1, \dotsc, t - 1\}$, each vertex in $B_i$ will have at most one
  neighbour in $B_{i+1} \cup \dotsb \cup B_t$.

  Let $\{B_i\}_{i \in [m]}$, for some $m \in \N$, be a maximal family of
  disjoint sets such that each $B_i$ induces a cycle of length at least four.
  Note that no matter how we specify the remaining sets, the desired degree
  property holds for $B_1, \dotsc, B_m$ due to $H$ having maximum degree three.

  Let $H' = H \setminus \bigcup_{i \in [m]} B_i$ and set $B_{m+1} \subseteq
  V(H')$ to be a largest subset which induces a path in $H'$. We aim to show
  that each vertex in $B_{m+1}$ has at most one neighbour in $R = V(H')
  \setminus B_{m+1}$. Suppose, towards a contradiction, that $v \in B_{m+1}$ has
  two neighbours $x, y \in R$, noting that $v$ must be an endpoint of the path
  $H'[B_{m+1}]$. Then both $x$ and $y$ necessarily have at least one more
  neighbour in $B_{m+1}$, as otherwise we could extend $H'[B_{m+1}]$ to a longer
  path. Let $v_x \in B_{m+1}$ be a neighbour of $x$ in $B_{m+1} \setminus \{v\}$
  which is closest to $v$ and define $v_y$ similarly, where the distance is
  measured along the path $H'[B_{m+1}]$. If $v v_x$ is not an edge, then the
  path from $v$ to $v_x$ in $H'[B_{m+1}]$ together with $xv$ and $xv_x$ forms an
  induced cycle of length at least four, contradicting the assumption that
  $H'$ has no such cycle. Therefore, $vv_x$ is an edge and, similarly, $vv_y$ is
  an edge. But then we must have that $v_x = v_y$ is the other
  endpoint of the path $H'[B_{m+1}]$ and so $H'[B_{m+1}]$ consists of a single
  edge. Note also that $xy$ is not an edge, as otherwise we would have that $H =
  K_4$. But then $\{x, v, y\}$ forms an induced path longer than $H'[B_{m+1}]$,
  contradicting the choice of $B_{m+1}$. Therefore, each vertex in $B_{m+1}$
  has at most one neighbour in $R$, as required.

  We now repeat this process, first by considering a longest induced path in
  $R$, until the entire vertex set is partitioned.
\end{proof}

An example of a partition of $V(H)$ as given by the procedure described above is
depicted in Figure~\ref{fig:H-partition}.

\begin{figure}[!htbp]
  \centering
  \includegraphics[scale=0.8]{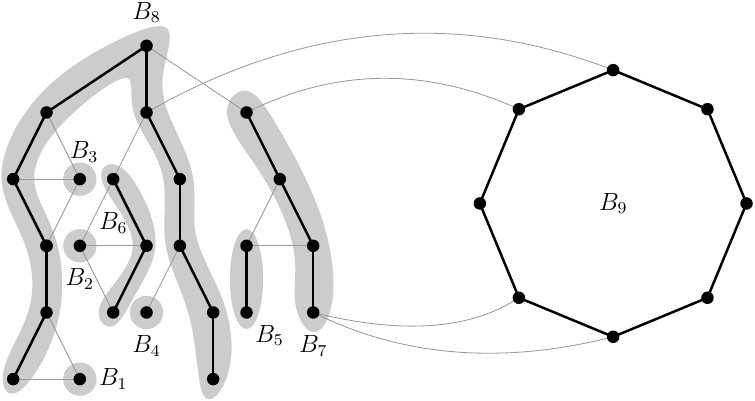}
  \caption{An example of a graph $H$ with a partition produced by
  Lemma~\ref{lem:decomposition}.}
  \label{fig:H-partition}
\end{figure}

We are now in a position to prove our main result. Recall the statement, that
there exist $c, K > 0$ such that if $p \ge Kn^{-2/5}$, then, with high
probability, $\Gnp \rightarrow H$ for every cubic graph $H$ with at most $cn$
vertices.

\begin{proof}[Proof of Theorem~\ref{thm:main_gnp}]
  As mentioned in Section~\ref{sec:proof_overview}, by following a standard
  reduction process, it suffices to show that $\Gamma \sim \Gnp$ with high
  probability satisfies the following: if a subgraph $G \subseteq \Gamma$
  contains disjoint sets of vertices $V_1, \dotsc, V_{20} \subseteq V(G)$, each
  of order $|V_i| = \tilde n = \alpha n$ for some constant $\alpha > 0$, such
  that each pair $(V_i, V_j)$ is $(\eps,p)$-regular of density at least $p/3$
  (with respect to $G$), then $G$ contains $H$. Here we take $\eps > 0$ to be as
  small as necessary for all our arguments to go through. This influences the
  choice of $\alpha$ and, consequently, decides how small $c$ has to be (recall
  that $H$ is a graph on $cn$ vertices), but the exact dependencies are not
  important. We can also assume that the density of each pair is exactly $d =
  p/3$ (see \cite[Lemma 4.3]{gerke2005sparse} or consider taking a random subset
  of the edges). We take this as our starting point.

  Using Lemma~\ref{lem:KLR}, we can find $cn$ vertex-disjoint copies of $K_4$ in
  $G$. Therefore, from now on we can assume that $H$ is $K_4$-free and
  connected. Let $B_1 \cup \dotsb \cup B_t = V(H)$ be a partition of $V(H)$ as
  given by Lemma~\ref{lem:decomposition} and let $\phi \colon V(H) \rightarrow
  [10]$ be an arbitrary colouring of $H$ such that if $v$ and $w$ are distinct
  vertices at distance at most two in $H$, then $\phi(v) \neq \phi(w)$.

  Remove from $G$ the `bad' set given by Lemma~\ref{lem:typical_vertices} for
  every choice of $(V_i, V_j)$ and, afterwards, take the subsets given by
  Lemma~\ref{lem:cleanup}. After `cleaning' the sets $V_1,\dotsc,V_{20}$ and
  subsequently renaming what remains, we are left with sets $\{V_i^0, V_i^1\}_{i
  \in [10]}$ such that $0.9 \tilde n \leq |V_i^*| \le \tilde n$, for $* \in
  \{0,1\}$, and the following properties hold for $(a,b) \in \{0,1\}^2$:
  \begin{enumerate}[label=(P{\arabic*}), leftmargin=3em]
    \item\label{emb-large-deg} $\deg_G(v, V_j^b) \geq \tilde n d/2$ for every
      $v \in V_i^a$ and all $i \neq j$;
    \item\label{emb-reg-inh} for each $i \neq j$ and any two vertices $v, w \not
      \in V_i^0 \cup V_i^1 \cup V_j^0 \cup V_j^1$, the pairs $(N_v, V_j^b)$ and
      $(N_v, N_w)$ are $(\eps', p)$-regular with density at least $d/2$ for any
      subsets $N_v \subseteq N_G(v, V_i^a)$ and $N_w \subseteq N_G(w, V_j^b)$ of
      order $\tilde n d / 4$.
  \end{enumerate}

  We sequentially embed the blocks $B_1, \dotsc, B_t$, where for each block $B_i$ we
  do the following:
  \begin{enumerate}[label=(Q{\arabic*}), leftmargin=3em]
    \item\label{proc-neighbour} For each $v \in B_i$, let $u_v \in V(G)$ be the
      image of its already embedded neighbour $a_v$ from $B_1 \cup \dotsb \cup
      B_{i-1}$. Without loss of generality, we can assume that $v$ always has
      such a neighbour.
    \item\label{proc-candidate} Choose the smallest $z_v \in \{0, 1\}$ such that
      $N_G\big(u_v, V_{\phi(v)}^{z_v}\big)$ has at least $\tilde n d / 4$
      vertices which are not already taken by the embedding of any vertex in
      $B_1 \cup \dotsb \cup B_{i-1}$. If no $z_v \in \{0,1\}$ has this property,
      terminate the procedure. Otherwise, let us denote the set of such `free'
      neighbours of $u_v$ by $F_v$.
    \item\label{proc-embed} As $|F_v| \ge \tilde n d / 4$ for each $v \in B_i$,
      property~\ref{emb-reg-inh} ensures that the conditions of
      Lemma~\ref{lem:tree-embedding-lemma} (`tree embedding') and
      Lemma~\ref{lem:cycle-embedding-lemma} (`cycle embedding') are satisfied,
      so there exists a copy of $H[B_i]$ in $G$ which maps each vertex $v \in
      B_i$ into its corresponding set $F_v$. Fix such an embedding of $H[B_i]$
      and proceed. (Minor technical detail: here we used the fact that if $w \in
      B_i$ is a neighbour of $v \in B_i$, then $\phi(a_v) \neq \phi(w)$, so that
      \ref{emb-reg-inh} can be applied for $F_v \subseteq V_{\phi(v)}^{z_v}$ and
      $F_w \subseteq V_{\phi(w)}^{z_w}$.)
  \end{enumerate}
  Assuming that the procedure did not terminate in step~\ref{proc-candidate}, we
  have successfully found a copy of $H$ in $G$. It therefore remains to show
  that the process does not terminate early. As in the proof of
  Lemma~\ref{lem:tree-embedding-lemma}, we will again use the first-free-bucket
  embedding strategy. However, the proof here is simpler as we only have two
  buckets to consider.

  Recall that $H$ has $cn$ vertices, so the set $X$ of all occupied vertices in
  $V_1^0 \cup \dotsb \cup V_{10}^0$ is of order at most $cn$. Let $Z_1$ denote
  the vertices $v \in V(H)$ for which $z_v = 1$ (as defined
  in~\ref{proc-candidate}) and let $U = \{u_v : v \in Z_1\}$. Then $|Z_1|$ is an
  upper bound on the number of occupied vertices in $V_1^1 \cup \dotsb \cup
  V_{10}^1$. We claim that $|Z_1| = O(1/p)$, which, together with the fact that
  each $u_v$ has at least $\tilde n d/2 \gg 1/p$ neighbours in $V_{\phi(v)}^1$
  (property~\ref{emb-large-deg}), implies that the procedure does not terminate
  in step \ref{proc-candidate}.

  Suppose, towards a contradiction, that $|Z_1| = 3C/p$ for a sufficiently large
  constant $C$ and so $|U| \geq C/p$. Then, by \ref{emb-large-deg} and
  \ref{proc-candidate}, we have that each vertex in $U$ has at least $\tilde n
  d/4$ neighbours in $X$. Hence, $e_G(U,X) \ge |U|\tilde n d/4$. On the other
  hand, for $C$ sufficiently large, a simple application of Chernoff's
  inequality and the union bound implies that w.h.p.~the number of edges in
  $\Gamma \sim \Gnp$ between any set $U$ of order at least $C/p$ and any set $X$
  of order $cn$ is at most $2|U||X|p$. Therefore, since $G \subseteq \Gamma$
  (and taking a superset of $X$ of sufficient size if necessary), $e_G(U, X) \le
  e_\Gamma(U, X) < 2|U|cnp$. Both the upper and lower bounds on $e_G(U, X)$ are
  of the same order of magnitude, but only the upper bound depends on $c$. Thus,
  for $c$ sufficiently small with respect to $\alpha$ (which is hidden in
  $\tilde n$), we get a contradiction.
\end{proof}

\section{Concluding remarks}
\label{sec:concluding}

As in \cite{kohayakawa2011sparse}, the proof of Theorem~\ref{thm:main_gnp}
easily extends to more than two colours and, more importantly, it actually gives
that in every $q$-colouring of $\Gnp$ one of the colours is \emph{universal} for
the family of cubic graphs with at most $cn$ vertices, that is, it contains all
such graphs simultaneously. It is known that any graph with this property has to
have $\Omega(n^{4/3})$ edges (see, for instance, \cite{alon08optimal}). In other
words, in order to go past that bound, one has to consider different host graphs
for different $H$. However, we are nowhere near that bound yet.

Recall that the argument for why $n^{8/5}$ is the best one can get from random
graphs is that, for $p = o(n^{-2/5})$, a typical $\Gnp$ is not Ramsey for $K_4$.
However, $K_4$ is the unique connected cubic graph with a fixed number of vertices which requires such a bound on $p$: for
any other such graph, $p =
\Theta(n^{-1/2})$ suffices (see \cite{rodl1993lower, rodl1995threshold}).
Therefore, it is not inconceivable that by adding some additional structure on
top of the random graph to deal with the $K_4$'s, one could obtain a bound of
type $n^{8/5 - \eps}$ or better.

Note that the only places where we used $p \ge Kn^{-2/5}$ in the proof of
Theorem~\ref{thm:main_gnp} were to invoke Lemma~\ref{lem:cycle-embedding-lemma}
to embed $C_4$ and to embed all components of $H$ isomorphic to $K_4$. In
particular, if $H$ does not contain $K_4$ and we can ensure in
Lemma~\ref{lem:decomposition} that all cycles are of length at least five, then
the same proof goes through with $p = \Theta(n^{-3/7})$. For longer cycles, this
bound decreases even further, being governed by the bound in
Lemma~\ref{lem:cycle-embedding-lemma}. Thus, we have the following improvements
over Theorem~\ref{thm:main-theorem}.

\begin{theorem}
  \label{thm:bipartite}
  There exists a constant $K$ such that $\hat r(H) \leq Kn^{11/7}$ for every
  triangle-free cubic graph $H$. Moreover, if $H$ is a cubic bipartite graph,
  then $\hat r(H) \leq Kn^{14/9}$.
\end{theorem}

\begin{proof}[Sketch of the proof]
  We first show that if $H$ is a connected triangle-free cubic graph with at
  least $7$ vertices, then it has a decomposition as in
  Lemma~\ref{lem:decomposition}, but where each cycle is now of length at least
  $5$. Following the proof of Lemma~\ref{lem:decomposition}, we start by
  removing from $H$ a maximal family of disjoint induced cycles of length at
  least $5$. Let $H'$ denote the resulting graph and, for each vertex $v \in
  V(H')$ of degree two (in $H'$), add a temporary new vertex which is joined
  only to $v$. Therefore, each vertex in this temporarily expanded $H'$ has
  degree either zero, one, or three. Consider a longest induced path in $H'$,
  say $v_1, \ldots, v_t$. We will show that it has the desired property that
  every $v_i$ has at most one neighbour in $V(H') \setminus \{v_1, \ldots,
  v_t\}$, in which case we can remove it from $H'$, together with all the
  temporary vertices, and proceed. Note that at least one vertex in $\{v_1,
  \ldots, v_t\}$ is an `original' vertex, so we indeed make progress.

  The only two vertices which can violate the desired property are $v_1$ and
  $v_t$. Without loss of generality, consider $v_1$ and suppose it has two
  neighbours $x, y \in V(H') \setminus \{v_2\}$. Then both of these vertices
  need to have a neighbour in $\{v_2, \ldots, v_t\}$, as otherwise we would get
  a longer induced path. Let us denote such a neighbour of $x$ closest to $v$ by
  $v_x$ and of $y$ by $v_y$. Then we necessarily have $v_x = v_y = v_3$, as
  otherwise we get an induced cycle of length at least $5$. Therefore, the
  longest induced path is of length two (i.e., $t = 3$). As $x$ has degree
  three, it must also have another neighbour $z$, which, as $H$ is
  triangle-free, lies outside of $\{y, v_1, v_2, v_3\}$
  (Figure~\ref{fig:triangle-free-partition}).

  \begin{figure}[!htbp]
  \centering
  \begin{tikzpicture}
    \draw[fill=black] (0,0) circle (2pt) node[above] {$x$};
    \draw[fill=black] (1.5,0) circle (2pt) node[above] {$v_1$};
    \draw[fill=black] (3,0) circle (2pt) node[above] {$v_2$};
    \draw[fill=black] (4.5,0) circle (2pt) node[above] {$v_3$};
    \draw[fill=black] (0,-1.5) circle (2pt) node[below] {$z$};
    \draw[fill=black] (1.5,-1.5) circle (2pt) node[below] {$y$};

    \draw (0, -1.5) -- (0,0) -- (1.5,0) -- (3,0) -- (4.5, 0) -- (1.5, -1.5) -- (1.5, 0);
    \draw[dashed] (3, 0) -- (0,-1.5) -- (1.5, -1.5);
    \draw (0,0) to [bend left] (4.5,0);
  \end{tikzpicture}
  \caption{Finishing the decomposition.}
  \label{fig:triangle-free-partition}
 \end{figure}
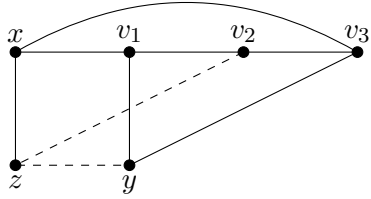

  Now we are almost done: if $yz$ is not an edge, then we get a longer induced
  path $y, v_1, x, z$, while if $zv_2$ is not an edge, then we get the path $z,
  x, v_1, v_2$. We must therefore have that $yz$ and $z v_2$ are both edges,
  which implies that $H'$, without the temporary vertices, is a connected graph
  with $6$ vertices where every vertex has degree three. But then $H' = H$,
  contradicting our assumption that $H$ has at least $7$ vertices.

  Therefore, there exists a decomposition with each cycle being of length at
  least $5$, as required. Moreover, if $H$ is bipartite, each such cycle has to
  have length at least $6$. To embed $H$, we then proceed by first embedding all
  connected components with at most $6$ vertices using Lemma~\ref{lem:KLR},
  which is possible already at $p = \Theta(n^{-1/2})$, as no such component is
  isomorphic to $K_4$. To embed cycles through subsets of order $\Theta(np)$
  using Lemma~\ref{lem:cycle-embedding-lemma}, we require $p =
  \Theta((np)^{-3/4})$ if they are of length at least $5$ and $p =
  \Theta((np)^{-4/5})$ if they are of length at least $6$, estimates which
  return the required bounds on $p$. The rest of the proof is then identical to
  the proof of Theorem~\ref{thm:main_gnp}.
\end{proof}

One can also obtain better bounds than those in \cite{kohayakawa2011sparse} if
$H$ has a special structure. For example, Clemens, Miralaei, Reding, Schacht,
and Taraz~\cite{clemens2021grid} showed that if $H$ is a $\sqrt{n} \times
\sqrt{n}$ grid, then $\hat r(H) = O(n^{3/2 + o(1)})$, which is essentially the
best one can get from random graphs in this case. Our methods immediately imply
this result, as the grid graph has a decomposition into induced paths (given by
`vertical' lines) which admit a `1-degenerate' ordering (i.e., the second
property of Lemma~\ref{lem:decomposition}). One notable feature of this argument
is that even though $H$ has many copies of $C_4$, we do not embed them directly
using Lemma~\ref{lem:cycle-embedding-lemma}.

Finally, it is worth noting that the problem of improving the bound from
\cite{kohayakawa2011sparse} for general bounded-degree graphs remains open. For
the class of triangle-free graphs $H$ with maximum degree $\Delta$, a rather
convoluted argument from the second author's PhD thesis \cite{rajko_thesis}
shows that $\hat{r}(H) = O(n^{2 - 1/\Delta - \eps_\Delta})$ for some
$\eps_\Delta > 0$. Regarding the methods presented here, analogues of
Lemmas~\ref{lem:tree-embedding-lemma} and~\ref{lem:cycle-embedding-lemma} with
each $N_v$ now contained in the neighbourhood of at most $\Delta-2$ vertices go
through provided $p = \Theta(n^{-1/(\Delta - 0.5)})$. Moreover, an easy
modification of the proof of Lemma~\ref{lem:decomposition} shows that a
connected graph $H$ with maximum degree $\Delta$, not isomorphic to
$K_{\Delta+1}$, admits a decomposition into blocks $B_1, \dotsc, B_t$ such that
each $B_i$ is either a path or a cycle of length at least four and each vertex
in $B_i$ has at most $\Delta-2$ neighbours in $B_1 \cup \dotsb \cup B_{i-1}$.
But, since these are the main ingredients of our proof in the $\Delta = 3$ case,
why does the proof not go through for larger $\Delta$?

The answer is somewhat technical, but the heart of the matter is that, for
$\Delta = 3$, no two vertices in $B_i$ can have a common neighbour in $B_{i+1}
\cup \dotsb \cup B_t$, which means that the inheritance properties promised by
Lemma~\ref{lem:typical_vertices} are sufficient to continue the embedding
process. If $\Delta > 3$, it can happen that two (or more) vertices in $B_i$
have such a common neighbour, which means that much more care is needed to
ensure that the candidate set for each vertex has the desired size and that pairs
inherit regularity in the correct way. It is entirely possible that this could
be done, but we have chosen not to pursue the matter here. Another, more
concrete, reason we did not pursue the matter further is that, if embedding
$K_{\Delta+1}$ is again the main obstacle, what one would really like to show is
that there are $c, K > 0$ such that if $p \geq K n^{-2/(\Delta+2)}$, then, with
high probability, $\Gnp \rightarrow H$ for any graph $H$ with at most $cn$
vertices and maximum degree $\Delta$. This is Theorem~\ref{thm:main_gnp} for
$\Delta = 3$, but it is almost certain that the techniques described here are
not sufficient to meet this bound for larger values of $\Delta$.

{\small \bibliographystyle{abbrv} \bibliography{large_rodl_rucinski}}

\end{document}